\documentclass{amsart}

\usepackage{amsmath} 
\usepackage{amssymb}

\newcommand{\forces}{\Vdash} 
\newcommand{\bV}{{\bf V}} 
\newcommand{\lesdot}{\mathrel{\mathord{<}\!\!\raise 
0.8 pt\hbox{$\scriptstyle\circ$}}} 
\newcommand{\supp}{{\rm supp}}

\newcommand{\cf}{{\rm cf}\/} 
\newcommand{\can}{{}^{\textstyle \omega}2} 
\newcommand{\fs}{{}^{\textstyle\omega{>}}2} 
\newcommand{\baire}{{}^{\textstyle \omega}\omega} 
 
\newcommand{\fsuo}{[\omega]^{\textstyle <\!\omega}} 
 
\newcommand{\lh}{\ell g\/} 
\newcommand{\rest}{{\restriction}}
\newcommand{\dom}{{\rm dom}} 
\newcommand{\rng}{{\rm rng}}

\newcommand{\ppn}{{\rm pnor}}
\newcommand{\nor}{{\rm nor}}
\newcommand{\val}{{\rm val}}
\newcommand{\rhalf}{{\rm half}}
\newcommand{\trunklg}{{\rm trunklg}}
\newcommand{\prodval}{\val^{\Pi}}
\newcommand{\conc}{{}^\frown\!}
\newcommand{\nF}{{\name{F}}}
\newcommand{\qfor}{{{\mathbb Q}^*_\infty(\bK,\bSig)}}
\newcommand{\qprod}{{{\mathbb P}_I(\bK,\bSig)}}
\newcommand{\qfl}{{{\mathbb P}_\lambda(\bK^*,\bSig^*)}}

\newcommand{\DO}{{\mathfrak D}{\mathfrak O}}
\newcommand{\fdo}{{\mathfrak d}{\mathfrak o}}
\newcommand{\bdo}{{\mathfrak d}_{{\mathcal B}}}

\newcommand{\bbA}{{\mathbb A}}
\newcommand{\cB}{{\mathcal B}}
\newcommand{\bbB}{{\mathbb B}}

\newcommand{\cC}{{\mathcal C}}
\newcommand{\gc}{{\mathfrak c}}
\newcommand{\gd}{{\mathfrak d}}

\newcommand{\cF}{{\mathcal F}}
\newcommand{\cI}{{\mathcal I}}

\newcommand{\cH}{{\mathcal H}}
\newcommand{\bH}{{\mathbf H}}
\newcommand{\bK}{{\mathbf K}}
\newcommand{\cK}{{\mathcal K}}
\newcommand{\cN}{{\mathcal N}}
\newcommand{\cM}{{\mathcal M}}
\newcommand{\bSig}{{\mathbf \Sigma}}
\newcommand{\cP}{{\mathcal P}}
\newcommand{\bbP}{{\mathbb P}}
\newcommand{\bQ}{{\mathbb Q}}  

\newcommand{\bP}{{\mathbb P}}

\newcommand{\mbR}{{\mathbb R}}
\newcommand{\cS}{{\mathcal S}}

\newcount\skewfactor
\def\mathunderaccent#1#2 {\let\theaccent#1\skewfactor#2
\mathpalette\putaccentunder}
\def\putaccentunder#1#2{\oalign{$#1#2$\crcr\hidewidth
\vbox to.2ex{\hbox{$#1\skew\skewfactor\theaccent{}$}\vss}\hidewidth}}
\def\name{\mathunderaccent\tilde-3 }

\newtheorem{theorem}{Theorem}[section] 
\newtheorem{claim}{Claim}[theorem]
\newtheorem{lemma}[theorem]{Lemma} 
\newtheorem{proposition}[theorem]{Proposition} 
\newtheorem{corollary}[theorem]{Corollary} 

\theoremstyle{definition}
\newtheorem{problem}[theorem]{Problem} 
 
\newtheorem{definition}[theorem]{Definition}

\theoremstyle{remark}

\newtheorem{remark}[theorem]{Remark}

\begin{document}

\dedicatory {Dedicated to L\'{a}szl\'{o} Fuchs for his ninetieth birthday}

\title{Monotone hulls for ${\mathcal N}\cap {\mathcal M}$}

\author{Andrzej Ros{\l}anowski}
\address{Department of Mathematics\\
 University of Nebraska at Omaha\\
 Omaha, NE 68182-0243, USA}
\email{roslanow@member.ams.org}
\urladdr{http://www.unomaha.edu/logic}

\author{Saharon Shelah}
\address{Institute of Mathematics\\
 The Hebrew University of Jerusalem\\
 91904 Jerusalem, Israel\\
 and  Department of Mathematics\\
 Rutgers University\\
 New Brunswick, NJ 08854, USA}
\email{shelah@math.huji.ac.il}
\urladdr{http://www.math.rutgers.edu/$\sim$shelah}

\thanks{Both authors acknowledge support from the United States-Israel
Binational Science Foundation (Grant no. 2006108). This is publication
972 of the second author.}

\date{July 15, 2014}

\begin{abstract}
  Using the method of decisive creatures (see Kellner and Shelah
  \cite{KrSh:872}) we show the consistency of ``there is no increasing
  $\omega_2$--chain of Borel sets and ${\rm non}(\cN)= {\rm non}(\cM)={\rm
    non}(\cN\cap\cM)=\omega_2=2^\omega$''. Hence, consistently, there are no
  monotone Borel hulls for the ideal ${\mathcal M}\cap {\mathcal N}$. This
  answers Balcerzak and Filipczak \cite[Questions 23, 24]{BaFi11}. Next we
  use finite support iteration of ccc forcing notions to show that there may
  be monotone Borel hulls for the ideals $\cM,\cN$ even if they are not
  generated by towers.
\end{abstract}

\maketitle

\section{Introduction}
Brendle and Fuchino \cite[Section 3]{BrFu07} considered the following
spectrum of cardinal numbers
\[\begin{array}{ll}
\DO=\big\{\cf({\rm otp}(\langle X, R\rest X\rangle)):& R\subseteq
\can\times\can \mbox{ is a projective binary relation,}\\
& X\subseteq \can\mbox{ and } R\cap X^2 \mbox{ is a well ordering of } 
X\big\} 
\end{array}\]
and they introduced a cardinal invariant $\fdo=\sup\DO$. The invariant $\fdo$
satisfies $\min\{{\rm non}(\cI),{\rm cov}(\cI)\}\leq \fdo$ for every ideal
$\cI$ on $\mbR$ with Borel basis (see \cite[Lemma 3.6]{BrFu07}). The proof
of Kunen \cite[Theorem 12.7]{Ku68} essentially shows that adding any number
of Cohen (or random) reals to a model of CH results in a model in which
$\fdo=\aleph_1$. Thus both 
\[\begin{array}{l}
{\rm non}(\cN)={\rm cov}(\cM)=\aleph_2+{\rm non}(\cM)={\rm cov}(\cN)=
\fdo=\aleph_1,\mbox{ and}\\ 
{\rm non}(\cM)={\rm cov}(\cN)=\aleph_2+{\rm non}(\cN)={\rm cov}(\cM)=
\fdo=\aleph_1
\end{array}\]
are consistent (where $\cM,\cN$ stand for the ideals of meager and null
sets, respectively). This naturally leads to the question if 
\begin{enumerate}
\item[$(\circledast)$] ${\rm non}(\cM)={\rm non}(\cN)={\rm non}(\cN\cap
  \cM)=\aleph_2+ \fdo=\aleph_1={\rm cov}(\cN)={\rm cov}(\cM)$
\end{enumerate}
is consistent. In this note we show the consistency of $(\circledast)$ using
the method of {\em decisive creatures\/} developed in Kellner and Shelah
\cite{KrSh:872}, and this method is in turn a special case of the method of
{\em norms on possibilities\/} of Ros{\l}anowski and Shelah \cite{RoSh:470}.  

Note that if there is a $\subset$--increasing $\kappa$--chain of Borel
subsets of $\can$, then $\cf(\kappa)\in\DO$. (Just consider a relation $R$
on $\can\simeq\can\times \can$ given by:\quad $(x,y)\; R\; (x',y')$ if and
only if `` $y,y'$ are Borel codes and $x$ belongs to the set coded by $y'$
''; cf. Elekes and Kunen \cite[Lemma 2.4]{ElKu03}.) Thus if we set 
\[\bdo=\sup\big\{\cf(\gamma):\mbox{ there is a $\subset$--increasing chain
  of Borel subset of $\mbR$ of length $\gamma$ }\big\}\]
then $\bdo\leq\fdo$. If $\bdo$ is smaller than the cofinality of the
uniformity number ${\rm non}(\cI)$ of a Borel ideal $\cI$, then there is no
monotone Borel hull operation on $\cI$ (see Elekes and M\'{a}th\'{e}  
\cite[Theorem 2.1]{ElMa09}, Balcerzak and Filipczak \cite[Theorem
5]{BaFi11}). Thus  
\begin{enumerate}
\item[$(\otimes)$] if $\cI$ is an ideal with Borel basis on $\mbR$, $\bdo<
  {\rm non}(\cI)$ and ${\rm non}(\cI)$ is a regular cardinal, then there is
  no $\subset$--monotone mapping $\psi:\cI\longrightarrow {\rm
    Borel}(\mbR)\cap \cI$.  
\end{enumerate}
Therefore in our model for $(\circledast)$ we will have (Corollary
\ref{cor3}) 
\smallskip

``there are no monotone Borel hull operations on the ideals $\cM,\cN$ and
$\cM\cap \cN$''.

\noindent This answers Balcerzak and Filipczak \cite[Question 23]{BaFi11}.
\medskip

We also obtain a positive result providing a new situation in which monotone 
hulls exist. Consistently, the ideals $\cM,\cN$ do not possess tower--basis
but they do admit monotone Borel hulls (Corollary \ref{cor4}). This model is
obtained by finite support iterations of partial Amoeba for Category and
Amoeba for Measure $\bbA$ forcing notions. 
\medskip

\noindent{\bf Notation}\qquad Most of our notation is standard and
compatible with that of classical textbooks (like Bartoszy\'nski and Judah
\cite{BaJu95}). However, in forcing we keep the older convention that {\em a
  stronger condition is the larger one}.  \medskip

$\bullet$ For two sequences $\eta,\nu$ we write $\nu\vartriangleleft\eta$
whenever $\nu$ is a proper initial segment of $\eta$, and
$\nu\trianglelefteq\eta$ when either $\nu\vartriangleleft\eta$ or
$\nu=\eta$. The length of a sequence $\eta$ is denoted by $\lh(\eta)$.
A {\em tree} is a family $T$ of finite sequences closed under initial
segments. For a tree $T$, the family of all $\omega$--branches through $T$
is denoted by $[T]$. 

$\bullet$ The Cantor space $\can$ is the space of all functions from
$\omega$ to $2$, equipped with the product topology generated by
sets of the form $[\nu]=\{\eta\in\can:\nu\vartriangleleft\eta\}$  for
$\nu\in\fs$. This space is also equipped with the standard product
measure $\mu$.

$\bullet$ For a forcing notion $\bP$, all $\bP$--names for objects in the
extension via $\bP$ will be denoted with a tilde below (e.g.~$\name{A}$,
$\name{\eta}$). The canonical name for a $\bbP$--generic filter over $\bV$
is denoted $\name{G}_\bbP$. Our notation and 
terminology concerning creatures and forcing with creatures will be
compatible with that in \cite{KrSh:872} (except of the reversed
orders). While this is a slight departure from the original terminology
established for creature forcing in \cite{RoSh:470}, the reader may find it
more convenient when verifying the results on decisive creatures that are
quoted in the next section.

\section{Background on decisive creatures}
As declared in the introduction, we will follow the notation and the context 
of \cite{KrSh:872} (which slightly differs from that of
\cite{RoSh:470}). For reader's convenience we will recall here all relevant
definitions and results from that paper.

Let $\bH:\omega\longrightarrow \cH(\aleph_0)$ (where $\cH(\aleph_0)$
is the family of all hereditarily finite sets). A {\em creating
  pair\/} for $\bH$ is a pair $(\bK,\bSig)$, where 
\begin{itemize}
\item $\bK=\bigcup\limits_{n<\omega}\bK(n)$, where each $\bK(n)$ is
  a finite set; elements of $\bK$ are called {\em creatures}, each
  creature $\gc\in\bK(n)$ has some norm $\nor(\gc)$ (a non-negative
  real number) and a non-empty set of possible values
  $\val(\gc)\subseteq \bH(n)$,  
\item if $\gc\in\bK(n)$, $\nor(\gc)>0$, then $|\val(\gc)|>1$
\item $\bSig:\bK\longrightarrow {\mathcal P}(\bK)$ is such that if
  $\gc\in\bK(n)$ and $\gc'\in\bSig(\gc)$, then $\gc'\in\bK(n)$, 
\item $\gc\in\bSig(\gc)$ and $\gc'\in \bSig(\gc)$ implies
  $\bSig(\gc')\subseteq \bSig(\gc)$, 
\item if $\gc'\in\bSig(\gc)$, then $\nor(\gc')\leq \nor(\gc)$ and
  $\val(\gc')\subseteq \val(\gc)$. 
\end{itemize}
If $\gc\in\bK$ and $x\in\bH(n)$, then we write $x\in\bSig(\gc)$ if and only
if $x\in\val(\gc)$. For $x\in\bH(n)$ we also set $\bSig(x)=\val(x)=\{x\}$. 

\begin{definition}
[See {\cite[Definitions 3.1, 4.1]{KrSh:872}}] 
\label{defbig}
  Let $0<r\leq 1$, $B,K,m$ be positive integers and $(\bK,\bSig)$ be a
  creating pair for $\bH$.
\begin{enumerate}
\item A creature $\gc$ is $r$--halving if there is a
  $\rhalf(\gc)\in\bSig(\gc)$ such that 
\begin{itemize}
\item $\nor(\rhalf(\gc))\geq \nor(\gc)-r$, and
\item if $\gd\in\bSig(\rhalf(\gc))$ and $\nor(\gd)>0$, then there is a
  $\gd'\in\bSig(\gc)$ such that
\[\nor(\gd')\geq \nor(\gc)-r\quad\mbox{ and } \quad
\val(\gd')\subseteq\val(\gd).\]  
\end{itemize}
$\bK(n)$ is $r$--halving, if all $\gc\in\bK(n)$ with $\nor(\gc)>1$ are
$r$--halving. 
\item A creature $\gc$ is $(B,r)$--big if for every function
  $F:\val(\gc)\longrightarrow B$ there is a $\gd\in\bSig(\gc)$ such that
  $\nor(\gd)\geq \nor(\gc)-r$ and the restriction $F\rest\val(\gd)$ is
  constant. We say that $\gc$ is hereditarily $(B,r)$-big, if every
  $\gd\in\bSig(\gc)$ with $\nor(\gd)>1$ is $(B,r)$-big. Also, $\bK(n)$ is
  $(B,r)$--big if every $\gc\in\bK(n)$ with $\nor(\gc)>1$ is $(B,r)$--big.
\item We say that $\gc$ is $(K,m,r)$--decisive, if for some $\gd^-,\gd^+\in
  \bSig(\gc)$ we have: $\gd^+$ is hereditarily $(2^{K^m},r)$--big, and
  $|\val(\gd^-)|\leq K$ and $\nor(\gd^-),\nor(\gd^+)\geq \nor(\gc)-r$. The
  creature $\gc$ is $(m,r)$--decisive if $\gc$ is $(K',m,r)$--decisive for
  some $K'$.
\item  $\bK(n)$ is $(m,r)$--decisive if every $\gc\in\bK(n)$
  with $\nor(\gc)>1$ is 
  $(m,r)$--decisive.  
\end{enumerate}
\end{definition}

\begin{lemma}
[See {\cite[Lemma 4.3]{KrSh:872}}] 
\label{lemdecisive}  
Assume that $(\bK,\bSig)$ is a creating pair for $\bH$, $k,m,t\geq 1$,
$0<r\leq 1$. Suppose that $\bK(n)$ is $(k,r)$--decisive and $\gc_0,
\ldots,\gc_{k-1}\in\bK(n)$ are hereditarily $(2^{m^t},r)$--big with
$\nor(\gc_i)>1+r\cdot (k-1)$ (for each $i<k$). Let
$F:\prod\limits_{i<k} \val(\gc_i)\longrightarrow 2^{m^t}$. Then
there are $\gd_0,\ldots,\gd_{k-1}\in \bK(n)$ such that: 

$\gd_i\in\bSig(\gc_i)$,\quad $\nor(\gd_i)\geq \nor(\gc_i)-r\cdot k$,\quad
and $F\rest\prod\limits_{i\in k} \val(\gd_i)$ is constant. 
\end{lemma}

A creating pair $(\bK,\bSig)$ determines the forcing notion $\qfor$
and its special product $\qprod$ as described by the following
definition. (The forcing notion $\qprod$ is a relative of the CS
product of $\qfor$ indexed by the set $I$.)  

\begin{definition}
[See {\cite[Definitions 2.1, 5.2, 5.3]{KrSh:872}}] 
\label{forcing}
\begin{enumerate}
\item   A condition in the forcing $\qfor$ is an $\omega$--sequence
  $p=\langle p(i):i<\omega\rangle$ such that for some $n<\omega$ (called the
  trunk-length of $p$) we have $p(i)\in \bH(i)$ if $i<n$, $p(i)\in
  \bK(i)$ and $\nor(p(i))>0$ if $i\geq n$, and
  $\lim\limits_{i\to\infty}(\nor(p(i)))=\infty$. 

The order on $\qfor$ is defined by $q\geq p$ if and only if (both
belong to $\qfor$ and) $q(i)\in\bSig(p(i))$ for all $i$.\footnote{Remember
  our convention that for $x,y\in\bH(i)$ and $\gc\in\bK(i)$ we write
  $x\in\bSig(\gc)$ iff $x\in\val(\gc)$, and  $x\in\bSig(y)$ iff $x=y$.}
\item Let $I$ be a non-empty (index) set. A condition $p$ in $\qprod$
  consists of a countable subset $\dom(p)$ of $I$, of objects $p(\alpha,n)$
  for $\alpha\in \dom(p)$, $n\in\omega$, and of a function
  $\trunklg(p,\cdot):\dom(p)\longrightarrow\omega$  satisfying the following
  demands for all $\alpha\in\dom(p)$:  
\begin{enumerate}
\item[$(\alpha)$] If $n<\trunklg(p,\alpha)$, then $p(\alpha,n)\in\bH(n)$. 
\item[$(\beta)$] If $n\geq \trunklg(p,\alpha)$, then $p(\alpha,n)\in\bK(n)$ 
      and $\nor(p(\alpha,n))>0$.
\item[$(\gamma)$] Setting $\supp(p,n)=\{\alpha\in\dom(p):\,
  \trunklg(p,\alpha)\leq n\}$, we have $|\supp(p,n)|<n$ for all $n>0$ and 
  $\lim\limits_{n\to\infty}(|\supp(p,n)|/n)=0$.
\item[$(\delta)$] $\lim\limits_{n\to\infty}(\min(\{\nor(p(\alpha,n)):\, 
  \alpha\in\supp(p,n)\}))= \infty$. 
\end{enumerate}
The order on $\qprod$ is defined by $q\geq p$ if and only if (both belong to
$\qprod$ and) $\dom(q)\supseteq \dom(p)$ and 
\begin{enumerate}
\item[$(\varepsilon)$] if $\alpha\in\dom(p)$ and $n\in\omega$, then
  $q(\alpha,n)\in\bSig(p(\alpha,n))$, 
\item[$(\zeta)$] the set $\{\alpha\in\dom(p):\trunklg(q,\alpha)\neq
  \trunklg(p,\alpha)\}$ is finite.
\end{enumerate}
\end{enumerate}
\end{definition}

Note that for $\alpha\in\dom(p)$ the sequence $\langle
p(\alpha,n):n\in \omega\rangle$ is in $\qfor$. However, $\qprod$ is
not a subforcing of the CS product of $I$ copies of $\qfor$ because of
a slight difference in the definition of the order relation. 

\begin{proposition}
[See {\cite[Lemmas 5.4, 5.5]{KrSh:872}}] 
\label{basprop}
\begin{enumerate}
\item If $J\subseteq I$, then $\bbP_J(K,\Sigma)=\{p\in
  \qprod:\dom(p)\subseteq J\}$ is a complete subforcing of $\qprod$. 
\item Assume {\rm CH}. Then $\qprod$ satisfies the $\aleph_2$--chain
  condition. 
\end{enumerate}
\end{proposition}

\begin{definition}
[See {\cite[Definition 5.6]{KrSh:872}}] 
\label{deciding}
\begin{enumerate}
\item For a condition $p\in\qprod$ we define\footnote{Remember
  our convention that, for $x\in\bH(i)$, $\val(x)=\{x\}$.}
\[\prodval(p,{<} n)=\prod\limits_{\alpha\in\dom(p)}
  \prod\limits_{m<n}\val(p(\alpha,m)).\] 
\item If $w\subseteq \dom(p)$ and $t\in \prod\limits_{\alpha\in w}
\prod\limits_{m<n}\bH(m)$, then $p\wedge t$ is defined by
\[\trunklg(p\wedge t,\alpha)=
  \begin{cases}
\max(\trunklg(p,\alpha),n)& \text{if }\alpha\in w,\\
\trunklg(p,\alpha) & \text{otherwise}
  \end{cases}\]
and
\[(p\wedge t)(\alpha,m)=
  \begin{cases}
  t(\alpha,m)& \text{if }m<n\text{ and }\alpha\in w,\\
 p(\alpha,m) & \text{otherwise.}
  \end{cases}\]
\item If $\name{\tau}$ is a name of an ordinal, then we say that $p$
  ${<}n$--decides $\name{\tau}$, if for every $t\in \prodval(p,{<} n)$
  the condition $p\wedge t$ forces a value to $\name{\tau}$. The
  condition $p$ essentially decides $\name{\tau}$, if $p$
  ${<}n$-decides $\name{\tau}$ for some $n$.
  \end{enumerate}
\end{definition}

\begin{proposition}
\begin{enumerate}
\item $p\wedge t\in \qprod$, and if $t\in \prodval(p,{<}n)$, then $p\wedge
  t\geq p$.
\item $\prodval(p,{<}n)\leq \prod\limits_{m<n} |\bH(m)|^m$.
\item $\{p\wedge t:\, t\in\prodval(p,{<}n)\}$ is predense above $p$
\end{enumerate}
\end{proposition}

\begin{theorem}
[See {\cite[Theorems 5.8, 5.9]{KrSh:872}}] 
\label{thm:prod}
Let $\varphi({<}n)=\prod\limits_{m<n} |\bH(m)|^m$ and
$0<r(n)\leq 1/(n^2\varphi({<}n))$.  Assume that each $\bK(n)$ is 
$(n,r(n))$-decisive and $r(n)$--halving (for $n\in\omega$). 
\begin{enumerate} 
\item The forcing notion $\qprod$ is proper and
  ${}^\omega\omega$-bounding. If $|I|\geq 2$ and $\lambda=
  |I|^{\aleph_0}$, then $\qprod$ forces $|I|\leq
  2^{\aleph_0}\leq\lambda$.  
\item Moreover, if $\name{\tau}(n)$ is a $\qprod$--name for an ordinal
  (for $n<\omega$) and $p\in\qprod$, then there is a condition $q\geq
  p$ which essentially decides all the names $\name{\tau}(n)$. 
\item Assume, additionally, that each $\bK(n)$ is $(g(n),r(n))$-- big,
  where $g\in\baire$ is strictly increasing. Suppose that
  $\name{\nu}(n)$ is a $\qprod$--name and $p\in\qprod$ forces that
  $\name{\nu}(n)<2^{g(n)}$ for all $n<\omega$. Then there is a $q\geq
  p$ which ${<}n$--decides $\name{\nu}(n)$ for all $n$.  
\end{enumerate}
\end{theorem}

The next theorem is a consequence of (the proof of) \cite[Corollaries
4.8(e), 3.9(b)]{BrFu07}. However, the results in \cite{BrFu07} are stated
for products, while $\qprod$ is not exactly a product (though it does have
all the required features). Therefore we will present the relatively simple
proof of this result fully. 

\begin{theorem}
\label{nochain}
Assume {\rm CH}. Let $r,\varphi,\bK$ and $\bSig$ be as in the assumptions of
Theorem \ref{thm:prod}. Then $\forces_{\qprod}\fdo=\bdo=\aleph_1$.
\end{theorem}

\begin{proof}
If $|I|\leq\aleph_1$, then $\forces_{\qprod} {\rm CH}$, so let us assume
$|I|\geq \aleph_2$.  

Every bijection $\pi:I\stackrel{\rm onto}{\longrightarrow} I$ determines an
automorphism $\tilde{\pi}$ of the forcing $\qprod$ in a natural way. Then,
for $J\subseteq I$, $\tilde{\pi}\rest {\mathbb P}_J(\bK,\bSig)$ is an
isomorphism from ${\mathbb P}_J(\bK,\bSig)$ onto ${\mathbb P}_{\pi[J]}
(\bK,\bSig)$. Also, $\pi$ gives rise to a natural bijection from
$\prodval(p,{<}n)$ onto $\prodval(\tilde{\pi}(p),{<}n)$; we will denote this
mapping by $\tilde{\pi}$ as well.

Suppose that $\varphi(x,y,\name{\tau})$ is a projective definition
of a binary relation on $\can$, where $\name{\tau}$ is a $\qprod$--name for
a real parameter. Assume towards contradiction that there are
$\qprod$--names $\name{\eta}_\alpha$ (for $\alpha<\omega_2$) and a condition 
$p\in\qprod$ such that
\begin{enumerate}
\item[(i)] $p\forces_{\qprod}$`` $\big(\forall\alpha,\beta<\omega_2\big)
  \big(\varphi(\name{\eta}_\alpha,\name{\eta}_\beta,\name{\tau})\
  \Leftrightarrow\ \alpha<\beta\big)$ ''. 
\end{enumerate}
For each $\alpha<\omega_2$ choose a condition $p_\alpha\geq p$ which
essentially decides all $\name{\eta}_\alpha(n)$ (for $n<\omega$). Then we
may also pick an increasing sequence $\bar{N}^\alpha= \langle N^\alpha_n:
n<\omega\rangle\subseteq \omega$ and a mapping
$f_\alpha:\bigcup\limits_{n<\omega}\prodval(p_\alpha,{<}N^\alpha_n)
\longrightarrow 2$ such that for each $t\in\prodval(p_\alpha,{<}N^\alpha_n)$
we have $(p_\alpha\wedge t)\forces\name{\eta}_\alpha(n)=f_\alpha(t)$. 

By CH, we may use a standard $\Delta$--system argument and the fact that
$\qprod$ satisfies the $\aleph_2$--cc (see \ref{basprop})  to choose $J\in
[I]^{\textstyle \aleph_1}$, $X\in [\omega_2]^{\textstyle \aleph_2}$ and
bijections $\pi_{\alpha,\beta}: \dom(p_\alpha)\stackrel{\rm
  onto}{\longrightarrow} \dom(p_\beta)$ such that   
\begin{enumerate}
\item[(ii)] $\dom(p)\subseteq J$ and $\name{\tau}$ is a 
${\mathbb P}_J(\bK,\bSig)$--name, and for distinct 
$\alpha,\beta\in X$:
\item[(iii)] $\dom(p_\alpha)\cap\dom(q_\beta)=\dom(p_\alpha)\cap J$ and
  $\pi_{\alpha,\beta}\rest (\dom(p_\alpha)\cap J)$ is the identity, 
\item[(iv)] $\tilde{\pi}_{\alpha,\beta}(p_\alpha)=p_\beta$,
  $\bar{N}^\alpha= \bar{N}^\beta$, and $f_\alpha=f_\beta\circ
  \tilde{\pi}_{\alpha,\beta}$. 
\end{enumerate}
Pick $\alpha<\beta$ from $X$. Let $\pi$ be a bijection from $I$ onto $I$
such that $\pi_{\alpha,\beta}\subseteq\pi$, $(\pi_{\alpha,\beta})^{-1}
\subseteq \pi$ and $\pi\rest J$ is the identity. Then 
\begin{enumerate}
\item[(v)] $\tilde{\pi}(p_\alpha)=p_\beta$, $\tilde{\pi}(p_\beta)
  =p_\alpha$ and $\tilde{\pi}(\name{\tau})=\name{\tau}$.  
\end{enumerate}
Note that $p_\alpha\cup p_\beta$ does not have to be a condition in $\qprod$
as the demand \ref{forcing}(2)($\gamma$) may fail. But extending finitely
many trunks will easily resolve this problem and we get a condition $q$
stronger than both $p_\alpha$ and $p_\beta$.  We may even do this in such a 
manner that the condition $q$ satisfies $\tilde{\pi}(q)=q$. Since $q\geq
p_\alpha,p_\beta$,  clause (iv) implies  
\begin{enumerate}
\item[(vi)] $q\forces$`` $\tilde{\pi}(\name{\eta}_\alpha)=
\name{\eta}_\beta\ \&\ \tilde{\pi}(\name{\eta}_\beta)=\name{\eta}_\alpha$
''. 
\end{enumerate}
Since $q\geq p$ and $\alpha<\beta$ we have $q\forces
\varphi(\name{\eta}_\alpha,\name{\eta}_\beta,\name{\tau})$. Applying the
automorphism $\tilde{\pi}$ and remembering (vi) we conclude that then also
$\tilde{\pi}(q)=q\forces \varphi(\name{\eta}_\beta,\name{\eta}_\alpha,
\name{\tau})$, contradicting (i).  
\end{proof}

\section{Consistency of $\fdo<{\rm non}(\cM\cap\cN)$}

\begin{definition}
\label{prenorm}
Let $n<\omega$.
\begin{enumerate}
\item A {\em basic $n$--block\/} is a finite non-empty set $B$ of functions
  from some non-empty $v\in\fsuo$ to 2 (i.e., $B\subseteq {}^{\textstyle v}
  2$) such that $|B|/2^{|v|}<2^{-n}$. If $\eta\in\fs\cup\can$ and
  $B\subseteq {}^{\textstyle v} 2$ is a basic block, then we write
  $\eta\prec B$ whenever $\eta\rest v\in B$. For an $n$--block $B\subseteq
  {}^{\textstyle v} 2$ we set $v(B)=v$.   
\item Let $H_n$ be the family of all pairs $(b,\cB)$ such that $b$ is a
  positive integer and $\cB$ is a non-empty finite set of basic 
  $n$--blocks. 
\item We define a function $\ppn:H_n\longrightarrow\omega$ by declaring 
  inductively when $\ppn(b,\cB)\geq k$. We set $\ppn(b,\cB)\geq 0$
  always, and then  
\begin{itemize}
\item $\ppn(b,\cB)\geq 1$ if and only if $(\forall F\in [\can]^{\textstyle
    b})(\exists B\in \cB)(\forall \eta\in F)(\eta\prec B)$,
\item $\ppn(b,\cB)\geq k+1$ if and only if there are positive integers 
  $b_0,\ldots,b_{M-1}$ and disjoint sets $\cB_0,\ldots,\cB_{M-1}
  \subseteq\cB$ such that  
\begin{enumerate}
\item[$(\alpha)$] $M>b^{k+1}$, $b_0\geq b$ and 
\item[$(\beta)$] $\ppn(b_i,\cB_i)\geq k$ and $(b_i)^2\cdot
  2^{|\cB_i|^n}<b_{i+1}$ for all $i<M$.  
\end{enumerate}
\end{itemize}
\end{enumerate}
\end{definition}

\begin{proposition}
\label{basicprop}
Let $n<\omega$, $(b,\cB),(b',\cB')\in H_n$.
\begin{enumerate}
\item $\ppn(b,\cB)\in \omega$ is well defined and $2^{\ppn(b,\cB)}\leq
  |\cB|$.  
\item If $\cB\subseteq \cB'$ and $b'\leq b$, then $\ppn(b,\cB) \leq
  \ppn(b',\cB')$.  
\item For each $N$ there is $(b^*,\cB^*)\in H_n$ such that 
\[b^*\geq N\mbox{ and }\ppn(b^*,\cB^*)\geq N\mbox{ and
}\min(v(B))>N\mbox{ for all }B\in\cB^*.\] 
\item If $\ppn(b,\cB)\geq k+1\geq 2$ and $c:\cB\longrightarrow \{0,\ldots,
  b-1\}$, then for some $\ell<b$ we have $\ppn(b,c^{-1}[\{\ell\}])\geq
  k$. 
\end{enumerate}
\end{proposition}

\begin{proof}
(1,2)\quad Easy induction on $\ppn(b,\cB)$.
\medskip

\noindent (3)\quad Note that if $w\in \fsuo$, $2^n\cdot N< 2^{|w|}$ and
$\cB_w$ consists of all basic $n$--blocks $B$ with $v(B)=w$, then
$\ppn(N,\cB_w)\geq 1$. Now proceed inductively.  
\medskip

\noindent (4)\quad Induction on $k\geq 1$. Assume $\ppn(b,\cB)\geq 2$
and $c:\cB\longrightarrow b$. We claim that for some $\ell<b$ we have
$\ppn(b,c^{-1}[\{\ell\}])\geq 1$. If not, then for each $\ell<b$ we
may choose $F_\ell\in [\can]^{\textstyle b}$ such that
\[\big(\forall B\in\cB\big)\big(\exists\eta\in
F_\ell\big)\big(c(B)=\ell\ \Rightarrow\ \eta\nprec B\big).\] 
Set $F=\bigcup\limits_{\ell<b} F_\ell$. Let $b_0,\ldots,b_{M-1},\cB_0, 
\ldots, \cB_{M-1}$ witness $\ppn(b,\cB)\geq 2$, in particular,
$b_1>b^2$ and $\ppn(b_1,\cB_1)\geq 1$. Since $|F|\leq b^2$ we conclude 
that there is $B\in\cB_1$ such that $(\forall\eta\in F)(\eta\prec
B)$. Then $B$ contradicts the choice of $F_{c(B)}$.

\noindent Now, for the inductive step, assume our statement holds
for $k$. Let $\ppn(b,\cB)\geq k+2$ and $c:\cB\longrightarrow
\{0,\ldots, b-1\}$. Let $\{(b_i,\cB_i):i<M\}$ witness $\ppn(b,\cB)\geq
(k+1)+1$, so $M>b^{k+2}$ and $\ppn(b_i,\cB_i) \geq k+1$ and $b_i\geq
b$. For each $i<M$ apply the inductive hypothesis to choose $\ell_i<b$
such that $\ppn(b_i,\cB_i\cap c^{-1}[\{ \ell_i\}])\geq k$. Choose $\ell^*<b$
such that $|\{i<M:\ell^*=\ell_i\}|> b^{k+1}$. Then $\{(b_i,\cB_i \cap
c^{-1}[\{ \ell_i\}]):\ell_i=\ell^*\}$ witnesses that
$\ppn(b,c^{-1}[\{\ell^*\}])\geq k+1$.   
\end{proof}

Now, by induction on $n<\omega$ we define the following objects
\begin{enumerate}
\item[$(\oplus)_n$] $\varphi_{\bH^*}({<}n), r_{\bH^*}(n), a(n), N_n, g(n),
  \bH^*(n), \bK^*(n),\bSig^*\rest \bK^*(n),\varphi_{\bH^*}(=n)$.  
\end{enumerate}
We start with stipulating $N_0=0$, $\varphi_{\bH^*}(<0)=1$.

\noindent Assume we have defined objects listed in $(\oplus)_k$ for
$k<n$, and that we also have defined integers $N_n,
\varphi_{\bH^*}({<}n)$. We set   
\begin{enumerate}
\item[(i)] $g(n)=2^{N_n}+\varphi_{\bH^*}({<}n)$, $r_{\bH^*}(n)=
  \frac{1}{(n+2)^2\varphi_{\bH^*}({<}n)}$ and $a(n)=2^{1/r_{\bH^*}(n)}$. 
\end{enumerate}
Choose $(b^*,\cB^*)\in H_n$ such that 
\begin{enumerate}
\item[(ii)] $b^*>g(n)$, $\min(v(B))>N_n$ for all
  $B\in\cB^*$ and $\ppn(b^*,\cB^*)>a(n)^{n+972}$
\end{enumerate}
(possible by \ref{basicprop}(3)). Set 
\begin{enumerate}
\item[(iii)] $N_{n+1}=\max\big(\bigcup\{v(B):B\in\cB^*\}\big)+1$.
\end{enumerate}
We let $\bH^*(n)$ be the set of all basic $n$--blocks $B$ such
that $v(B)\subseteq [N_n,N_{n+1})$, and $\bK^*(n)$ consist of all triples 
$\gc=(k^\gc,b^\gc,\cB^\gc)$ such that 
\[(b^\gc,\cB^\gc)\in H_n,\ \ \cB^\gc\subseteq \bH^*(n),\ \ b^\gc>g(n),\
\mbox{ and }\ k^\gc\in \omega,\ \ k^\gc<\ppn(b^\gc,\cB^\gc)-1.\] 
For $\gc\in\bK^*(n)$ we set
\begin{enumerate}
\item[(iv)]  $\nor(\gc)=\log_{a(n)}\big(\ppn(b^\gc,
\cB^\gc)-k^\gc\big)$, $\val(\gc)=\cB^\gc$ and\\ 
$\bSig^*(\gc)=\{\gd\in\bK^*(n): k^\gc\leq k^\gd,\ b^\gc\leq b^\gd,\ 
\cB^\gd\subseteq \cB^\gc\}$. 
\end{enumerate}
Finally, we put $\varphi_{\bH^*}({=}n)=|\bH^*(n)|^n$ and
$\varphi_{\bH^*}({<}n+1)=\varphi_{\bH^*}({<}n)\cdot
\varphi_{\bH^*}({=}n)$. This completes our inductive definition.

\begin{proposition}
\label{pairOK}
$(\bK^*,\bSig^*)$ is a creating pair for $\bH^*$ such that, for each 
$n<\omega$, $\bK^*(n)$ is $(n,r_{\bH^*}(n))$--decisive,
$r_{\bH^*}(n)$--halving and $(g(n),r_{\bH^*}(n))$--big. 
\end{proposition}

\begin{proof}
To verify halving, for each $\gc\in \bK^*(n)$ with $\nor(\gc)>1$ set
\[\rhalf(\gc)=(k^\gc+\lfloor \frac{1}{2}(\ppn(b^\gc,\cB^\gc)-k^\gc)\rfloor,
b^\gc,\cB^\gc).\]
Note that $\nor(\gc)>1$ implies $\ppn(b^\gc,\cB^\gc)-k^\gc>2$ and hence
\[k^\gc+ \lfloor \frac{1}{2}(\ppn(b^\gc,\cB^\gc)-k^\gc)\rfloor <\ppn(b^\gc,
\cB^\gc)-1.\]
Therefore, $\rhalf(\gc)\in\bSig^*(\gc)$ and $\nor(\rhalf(\gc)) \geq
\nor(\gc)-r_{\bH^*}(n)$. Now suppose $\gd\in\bSig^*(\rhalf(\gc))$, so
$k^\gc+\lfloor \frac{1}{2}(\ppn(b^\gc,\cB^\gc)-k^\gc)\rfloor\leq k^\gd$,
$b^\gc\leq b^\gd$ and $\cB^\gd\subseteq \cB^\gc$. Also, $k^\gd<
\ppn(b^\gd,\cB^\gd)-1$, so $\ppn(b^\gd,\cB^\gd)> k^\gc+ \lfloor
\frac{1}{2}(\ppn(b^\gc,\cB^\gc)-k^\gc)\rfloor+1$. Consider
$\gd'=(k^\gc,b^\gd,\cB^\gd)$. Plainly $\gd'\in\bSig^*(\gc)$, $\val(\gd')
\subseteq \val(\gd)$ and
\[\begin{array}{r}
\nor(\gd')\geq \log_{a(n)} \big(\lfloor \frac{1}{2}(\ppn(b^\gc, \cB^\gc)-
k^\gc) \rfloor+1\big)\geq \log_{a(n)}\big(\frac{1}{2}(\ppn(b^\gc, \cB^\gc)-
k^\gc)\big)\\
=\nor(\gc)-r_{\bH^*}(n).
\end{array}\]

It follows from \ref{basicprop}(4) that 
\begin{enumerate}
\item[$(*)$] if $\gc\in\bK^*(n)$, $\nor(\gc)>r_{\bH^*}(n)$, then $\gc$ is
  $(b^\gc,r_{\bH^*}(n))$--big. 
\end{enumerate}
Hence $\bK^*(n)$ is $(g(n),r_{\bH^*}(n))$--big (remember the definition of
$\bK^*(n)$).  

Now suppose $\gc\in\bK^*(n)$, $\nor(\gc)>1$. Then $\ppn(b^\gc,\cB^\gc)-
k^\gc>2$, so by the definition of $\ppn$ (see \ref{prenorm}(3)) we may find
$b^\gc\leq b_0<b_1<\ldots<b_{M-1}$ and disjoint $\cB_0,\ldots,\cB_{M-1}
\subseteq \cB^\gc$ such that $\ppn(b_i,\cB_i)\geq \ppn(b^\gc,\cB^\gc)-1$ and
$(b_i)^2\cdot 2^{|\cB_i|^n}<b_{i+1}$. Set 
\[\gd^-=(k^\gc,b_0,\cB_0),\quad \gd^+=(k^\gc,b_1,\cB_1),\quad \mbox{ and }
K=|\cB_0|.\]
Plainly, $\gd^-,\gd^+\in\bSig(\gc)$, $\min\{\nor(\gd^-),\nor(\gd^+)\} \geq
\nor(\gc)-r_{\bH^*}(n)>r_{\bH^*}(n)$ and $|\val(\gd^-)|=K$. Also $\gd^+$ is
hereditarily  $(2^{K^n},r_{\bH^*}(n))$--big (remember $b_1>2^{K^n}$, use
$(*)$). Thus $\gd^-,\gd^+$ witness that $\gc$ is
$(K,n,r_{\bH^*}(n))$--decisive.  
\end{proof}

\begin{definition}
\begin{enumerate}
\item For a cardinal $\lambda$ we consider the forcing notion $\qfl$
determined by the creating pair $(\bK^*,\bSig^*)$ as in
\ref{forcing}(2). For $\alpha<\lambda$, a $\qfl$--name $\name{\rho}_\alpha$
is defined by 
\[\forces_{\qfl}\name{\rho}_\alpha=\bigcup\big\{p(\alpha,n):
\alpha\in\dom(p)\ \&\ n<\trunklg(p,\alpha)\ \&\ p\in\name{G}_{\qfl}\big\}.\]   
\item For $\rho\in\prod\limits_{n<\omega}\bH^*(n)$ we set $F(\rho)=\big\{
  \eta\in\can:(\forall^\infty n<\omega)(\eta\prec\rho(n)\big\}$.
\end{enumerate}
\end{definition}

Plainly, for each $\alpha<\lambda$, $\forces_{\qfl}\name{\rho}_\alpha \in
\prod\limits_{n<\omega}\bH^*(n)$. Also, for $\rho\in \prod\limits_{n<\omega}
\bH^*(n)$, the set $F(\rho)$ is a meager and null $\bSig^0_2$--subset of
$\can$. 

\begin{theorem}
\label{forcingOK}
Assume CH. Let $\lambda$ be an uncountable cardinal,
$\lambda=\lambda^{\aleph_0}$.
\begin{enumerate}
\item Forcing with $\qfl$ preserves cardinalities and cofinalities and\\
  $\forces_{\qfl}$`` $2^{\aleph_0}=\lambda$ ''.
\item If $\beta<\lambda$ and $\name{\nu}$ is a ${\mathbb
    P}_{\lambda\setminus\{\beta\}}(\bK^*,\bSig^*)$--name for a
member of $\can$, then $\forces_{\qfl}$`` $\name{\nu}\in
F(\name{\rho}_\beta)$ ''.
\item Consequently, $\forces_{\qfl}$`` ${\rm non}({\mathcal N})= {\rm 
    non}({\mathcal M})=\lambda$ ''.
\end{enumerate}
\end{theorem}

\begin{proof}
(1)\quad It follows from \ref{pairOK}+\ref{basprop}(2)+\ref{thm:prod}.
\medskip

\noindent (2)\quad The proof is parallel to that of \cite[Lemma
9.1]{KrSh:872}. Assume $p\in\qfl$. Remembering \ref{basprop}(1) we
may use \ref{thm:prod}(3) to find a condition $q\geq p$ such that
\begin{enumerate}
\item[$(*)_1$] the condition $q\rest (\lambda\setminus\{\beta\})$
  ${<}n$--decides the value of $\name{\nu}\rest N_n$ (for each $n$),
  and   
\item[$(*)_2$] $\trunklg(q,\alpha)\geq 972$ for all $\alpha\in\dom(q)$ and 
  $\nor(q(\alpha,m))\geq 972$ whenever $\alpha\in \supp(q,m)$, and
\item[$(*)_3$] $\beta\in\dom(q)$ and if $\supp(q,m)\neq\emptyset$, then
  $|\supp(q,m)|\geq 972$. 
\end{enumerate}
Thus, for each $n$, we have a mapping $E_n:\prodval(q \rest
(\lambda\setminus\{\beta\}),{<}n)\longrightarrow {}^{\textstyle N_n}2$ such
that 
\[\big(q\rest (\lambda\setminus\{\beta\})\big) \wedge
t\forces_{{\mathbb P}_{\lambda\setminus\{\beta\}}(\bK^*,\bSig^*)}\mbox{`` }
\name{\nu}\rest N_n=E_n(t)\mbox{ ''.}\]
We will further strengthen $q$ to a condition $q^*$ such that
$\dom(q^*)=\dom(q)$ and  
\begin{enumerate}
\item[$(*)^{\rm goal}$] for all $n\geq \trunklg(q^*,\beta)$ and $t\in
\prodval \big(q^*\rest (\lambda\setminus\{\beta\}),{<}(n+1)\big)$ we have 
\[\big(\forall B\in q^*(\beta,n)\big)\big(E_{n+1}(t)\prec B\big).\]  
\end{enumerate}
Then clearly we will have $q^*\forces_{\qfl}\mbox{`` }\name{\nu}\in
F(\name{\rho}_\beta)\mbox{ ''}$ and the proof of \ref{forcingOK}(2) will
follow by the standard density argument.

To construct the condition $q^*$ we set $\dom(q^*)=\dom(q)$,
$\trunklg(q^*,\alpha)=\trunklg(q,\alpha)$, and we define $q^*(\alpha,m)$ by
induction on $m$ so that: 

$q^*(\alpha,m)=q(\alpha,m)$ whenever $\alpha\notin\supp(q,m)$ or
$\beta\notin\supp(q,m)$, and 

$q^*(\alpha,m)\in \bSig^*(q(\alpha,m))$, $\nor(q^*(\alpha,m))\geq
  \nor(q(\alpha,m))-2$ for $\alpha\in\supp(q,m)$.

\noindent These demands guarantee that $q^*$ is a condition in $\qfl$
stronger than $q$. 

Fix an $n\geq\trunklg(q,\beta)$. Put $A=\supp(q,n)$ and note that 
$\beta\in A$, $A$ has at least 972 elements (remember $(*)_3$), and
$|A|<n$ (by \ref{forcing}(2)($\gamma$)). 

Set $\gc^0_\alpha=q(\alpha,n)$ for $\alpha\in A$. 

We choose inductively an enumeration $\{\alpha_0,\ldots,\alpha_{|A|-1}\}$ of
$A$ and creatures $\gc^\ell_{\alpha_k}$ (for $\ell\leq k< |A|$) and 
$\gd_{\alpha_k}$ from $\bSig^*(\gc^0_{\alpha_k})$. So assume that for some
$\ell\geq 0$ we already have defined a list $\{\alpha_k:k<\ell\}$ of
distinct elements of $A$ and creatures $\gc^\ell_\alpha$ for $\alpha\in
A\setminus \{ \alpha_0,\ldots, \alpha_{\ell-1}\}$. Each $\gc^\ell_\alpha$ is
$(K^\ell_\alpha,n, r_{\bH^*}(n))$--decisive for some $K^\ell_\alpha$. Put
$K_\ell= \min(\{K^\ell_\alpha:\, \alpha\in A\setminus \{\alpha_0, \ldots,
\alpha_{\ell-1}\}\})$, and choose $\alpha_\ell$ such that
$K_{\alpha_\ell}^\ell=K_\ell$. Let $\gd_{\alpha_\ell}\in
\bSig^*(\gc^\ell_{\alpha_\ell})$ be such that $|\val(\gd_{\alpha_\ell})|\leq
K_\ell$ and $\nor(\gd_{\alpha_\ell})\geq\nor(\gc^\ell_{\alpha_\ell})-
r_{\bH^*}(n)$. For $\alpha\in A\setminus \{\alpha_0,\ldots,\alpha_\ell\}$,
let $\gc^{\ell+1}_\alpha\in\bSig^*(\gc^\ell_\alpha)$ be hereditarily
$(2^{(K_\ell)^n},r_{\bH^*}(n))$--big and such that $\nor(\gc^{\ell+1}_\alpha
) \geq\nor(\gc^\ell_{\alpha_\ell})-r_{\bH^*}(n)$. Iterate this procedure
$|A|-1$ times. At the end, there remains one $\alpha$ that has not been
listed as an $\alpha_\ell$, so we set $\alpha_{|A|-1}=\alpha$ and
$\gd_{\alpha_{|A|-1}}=\gc^{|A|-1}_\alpha$.
  
Since $\gc^{\ell+1}_{\alpha_{\ell+1}}$ is hereditarily $(2^{(K_\ell)^n},
r_{\bH^*}(n))$--big, we see that $2^{(K_\ell)^n}<K_{\ell+1}$. Let $m$ be
such that $\beta=\alpha_m$, and put
\[K=K_m,\quad S=\{\alpha_\ell:\, \ell<m\},\quad L=\{\alpha_\ell:\,
\ell>m\}.\]
It is possible that (at most) one of the sets $S,L$ is empty. By our
choices, 
\begin{enumerate}
\item[$(*)_4$]
\begin{enumerate}
\item[(a)] $\gd_\alpha\in \bSig^*(q(\alpha,n))$, $\nor(\gd_\alpha)\geq 
  \nor(q(\alpha,n))-(n-1)\cdot r_{\bH^*}(n)>900$, and 
\item[(b)] if $S\neq\emptyset$ then $\gd_\beta$ is $(2^{(K_{m-1})^n},
  r_{\bH^*}(n))$--big and hence in particular $(K_{m-1})^{n-2}<K$; if
  $S=\emptyset$ then $K=K_0$,  
\item[(c)] $\prod\limits_{\alpha\in S}|\val(\gd_\alpha)|\leq
  (K_{m-1})^{n-2}<K$ and $|\val(\gd_\beta)|\leq K$,
\item[(d)] $\varphi_{\bH^*}({<}n)<K_0\leq K$ (remember that $\bK(n)$ is
  $(g(n),r_{\bH^*}(n))$--big and $g(n)>\varphi_{\bH^*}({<}n)$),
\item[(e)] if $\alpha\in L$, then $\gd_\alpha$ is
  $(2^{K^n},r_{\bH^*}(n))$--big. 
\end{enumerate}
\end{enumerate}

Let $Z=\{t\in \prodval(q \rest (\lambda\setminus\{\beta\}), {<}(n+1)):
t(\alpha,n)\in\val(\gd_\alpha)\mbox{ for }\alpha\in A\setminus\{\beta\} \}$
and for $s\in \prod\limits_{\alpha\in L}\val(\gd_\alpha)$ let $Z_s=\{t\in Z:
t(\alpha,n) =s(\alpha)\mbox{ for }\alpha\in L\}$. Next, for $t\in Z$ put 
$\cC_t=\{B\in\cB^{\gd_\beta}:E_{n+1}(t)\nprec B\}$. 
\medskip

If $S=\emptyset$, then in what follows ignore $\prod\limits_{\alpha\in
  S}\val(\gd_\alpha)$ and set $K_{m-1}=1$. Assume  $L$ is non-empty
(otherwise move to $(*)_6$). For each $s\in \prod\limits_{\alpha\in 
  L}\val(\gd_\alpha)$ consider a function  
\[c(s):\prodval(q \rest (\lambda\setminus\{\beta\}), {<}n) \times
\prod\limits_{\alpha\in S}\val(\gd_\alpha)\longrightarrow
\cP(\val(\gd_\beta))\] 
such that $c(s)(t_0,t_1)=\cC_{t_0\conc t_1\conc s}$, where $t_0\conc
t_1\conc s\in Z_s$ is obtained by natural concatenation. This determines a
coloring $c$ on $\prod\limits_{\alpha\in L}\val(\gd_\alpha)$ with the range
of size at most 
\[\left( 2^K\right)^{\varphi_{\bH^*}({<}n)\cdot (K_{m-1})^{n-2}}\leq
\left( 2^K\right)^{K\cdot K}=2^{K^3}<2^{K^n}.\]
Since $\bK^*(n)$ is $(n,r_{\bH^*}(n))$--decisive, and each $\gd_\alpha$  is 
hereditarily $(2^{K^n},r_{\bH^*}(n))$--big (for $\alpha\in L$), 
$\nor(\gd_\alpha)>900$ and $|L|\leq n-2$, therefore we may use Lemma
\ref{lemdecisive} to find $q^*(\alpha,n)\in\bSig^*(\gd_\alpha)$ for
$\alpha\in L$ such that 
\begin{enumerate}
\item[$(*)_5$]
\begin{enumerate}
\item[(a)] $\nor(q^*(\alpha,n))\geq \nor(\gd_\alpha)-r_{\bH^*}(n)\cdot n\geq
  \nor(q(\alpha,n))-2$, and
\item[(b)] $c\rest \prod\limits_{\alpha\in L}\val(q^*(\alpha,n))$ is
  constant. 
\end{enumerate}
\end{enumerate}
If $L=\emptyset$ then the procedure described above is not needed. In any
case, letting 
\[X=\prodval(q \rest (\lambda\setminus\{\beta\}), {<}n) \times
\prod\limits_{\alpha\in S}\val(\gd_\alpha),\]
we have a mapping $d:X\longrightarrow \cP(\val(\gd_\beta))$ and
$q^*(\alpha,n)$ for $\alpha\in L$ such that  
\begin{enumerate}
\item[$(*)_6$] if $t\in Z$ and $t(\alpha,n)\in \val(q^*(\alpha,n))$ for
  $\alpha\in L$, then $\cC_t=d(t_0,t_1)$, where $t_0=t\rest
  \big((\dom(q)\setminus \{\beta\})\times n\big)\in \prodval(q \rest
  (\lambda\setminus\{\beta\}), {<}n)$ and $t_1=t\rest (S\times\{n\})\in
  \prod\limits_{\alpha\in S}\val(\gd_\alpha)$.
\end{enumerate}
For each $(t_0,t_1)\in X$ fix one $t=t[t_0,t_1]\in Z$ such that
$t(\alpha,n)\in \val(q^*(\alpha,n))$ for $\alpha\in L$,
$t_0=t\rest\big((\dom(q)\setminus \{\beta\})\times n\big)$ and $t_1=t\rest 
(S\times\{n\})$. Now, for $B\in\val(\gd_\beta)$ we (try to) choose
$(t^B_0,t^B_1)\in X$ such that $B\in\cC_{t[t^B_0,t^B_1]}$, if
possible. Consider a coloring $e:\val(\gd_\beta)\longrightarrow 
{}^{N_{n+1}}2\cup\{*\}$ defined by 
\[e(B)=\left\{\begin{array}{ll}
E_{n+1}(t[t^B_0,t^B_1])&\mbox{ if $(t^B_0,t^B_1)\in X$ is defined,}\\
*           &\mbox{ otherwise.}
\end{array}\right.\]
Since $|X|\leq \varphi_{\bH^*}({<}n)\cdot (K_{m-1})^{n-2}\leq\max
\{(K_{m-1})^{n-1}, \varphi_{\bH^*}({<}n)\}$, we know that the range of the
coloring $e$ has at most $\max\{(K_{m-1})^{n-1}, \varphi_{\bH^*}({<}n)\}+1$  
members. Thus $\gd_\beta$ is $(|\rng(e)|,r_{\bH^*}(n))$--big and we may
choose $q^*(\beta,n)\in\bSig^*(\gd_\beta)$ such that $\nor(q^*(\beta,n))\geq
\nor(\gd_\beta)-r_{\bH^*}(n)\geq \nor(q(\alpha,n))-2>900$ and $e\rest
\val(q^*(\alpha,n))$ is constant. If the constant value were $\eta\in
{}^{N_{n+1}}2$, then we would have $\eta\nprec B$ for all $B\in
\val(q^*(\alpha,n))$, contradicting $\nor(q^*(\beta,n))>0$. Therefore,
\begin{enumerate}
\item[$(*)_7$] $(t^B_0,t^B_1)$ is defined for no $B\in\val(q^*(\beta,n))$
  and hence 
\[\val(q^*(\beta,n))\cap \bigcup\{\cC_{t[t_0,t_1]}:(t_0,t_1)\in X\}=
\emptyset.\]   
\end{enumerate}
For $\alpha\in S$ we set $q^*(\alpha,n)=\gd_\alpha$. Now note that 
\begin{enumerate}
\item[$(*)_8$] if $t\in Z$ is such that $t(\alpha,n)\in q^*(\alpha,n)$
  for $\alpha\in S\cup L$ and $B\in \val(q^*(\beta,n))$, then
  $E_{n+1}(t)\prec B$. 
\end{enumerate}
Why? Assume towards contradiction that $E_{n+1}(t)\nprec B$, i.e., $B\in
\cC_t$. Represent $t$ as $t=t_0\conc t_1\conc s$ where $(t_0,t_1)\in
X$. Then $\cC_t=\cC_{t[t_0,t_1]}$ (by $(*)_6$) and therefore $B\in
\cC_{t[t_0,t_1]}$, contradicting $(*)_7$.
\medskip

\noindent This completes the definition of $q^*$. It follows from $(*)_8$
(for $n\geq\trunklg(q^*,\beta)$) that $(*)^{\rm goal}$ is satisfied. 
\medskip
 
\noindent (3)\quad Follows from (2) and the fact that $F(\rho)\in \cN\cap
\cM$ for $\rho\in\prod\limits_{n<\omega}\bH(m)$. 
\end{proof}

\begin{corollary}
\label{cor1}
It is consistent that 
\[{\rm non}(\cN)={\rm non}(\cM)={\rm non}(\cN\cap\cM)=\aleph_2=2^{\aleph_0}
\mbox{ and }\fdo=\aleph_1.\]
\end{corollary}

\begin{proof}
Start with a model of CH and force with ${\mathbb
  P}_{\aleph_2}(\bK^*,\bSig^*)$. It follows from \ref{forcingOK} and
\ref{nochain} that the resulting model is as required.
\end{proof}

In models for the statement in Corollary \ref{cor1} necessarily ${\rm
  cov}(\cN)={\rm cov}(\cM)=\aleph_1$. However, it is not clear if we could
not get a parallel result for $\bdo$ and ${\rm cov}$.

\begin{problem}
Is it consistent that 
\[{\rm cov}(\cN)={\rm cov}(\cM)=\aleph_2=2^{\aleph_0}
\mbox{ and }\bdo=\aleph_1\mbox{ ?}\]
In particular, is it consistent that $\fdo>\bdo$ ?
\end{problem}

Directly from \ref{cor1} we also obtain
\begin{corollary}
\label{cor2}
It is consistent that ${\rm non}(\cN\cap\cM)=\aleph_2$ and there is no
$\subset$--increasing chain of Borel subset of $\can$ of length $\omega_2$. 
\end{corollary}

\section{Monotone hulls}
The interest in Corollary \ref{cor2} came from the questions concerning
Borel hulls. 

\begin{definition}
Let ${\rm Borel}(\can)$ be the family of all Borel subsets of $\can$,
$\cI$ be a $\sigma$--ideal on $\can$ with Borel basis and $\cS_\cI$ be the
$\sigma$--algebra of subsets of $\can$ generated by ${\rm Borel}(\can)\cup
\cI$. Let $\cF\subseteq \cS_\cI$. A {\em monotone Borel hull\/} on $\cF$
with respect to $\cI$ is a mapping $\psi:\cF\longrightarrow {\rm
  Borel}(\can)$ such that 
\begin{itemize}
\item $A\subseteq \psi(A)$ and $\psi(A)\setminus A\in\cI$ for all $A\in\cF$,
  and 
\item if $A_1\subseteq A_2$ are from $\cF$, then $\psi(A_1)\subseteq
  \psi(A_2)$. 
\end{itemize}
If the range of $\psi$ consists of sets of some Borel class $\cK$, then we  
say that $\psi$ is a monotone $\cK$ hull operation.
\end{definition}

\noindent As discussed in Balcerzak and Filipczak \cite[Question
24]{BaFi11}, \ref{cor2} implies the following.

\begin{corollary}
\label{cor3}
It is consistent that
\begin{itemize}
\item there are no monotone Borel hulls on $\cM$ with respect to $\cM$, and 
\item there are no monotone Borel hulls on $\cN$ with respect to $\cN$, and 
\item there are no monotone Borel hulls on $\cM\cap \cN$ with respect to
  $\cM\cap\cN$. 
\end{itemize}
\end{corollary}

The non-existence of monotone Borel hulls on $\cI$ implies non-existence of 
such hulls on $\cS_\cI$. While some partial results were presented in 
\cite{ElMa09} and \cite{BaFi11}, not much is known about the converse
implication. 

\begin{problem}
[Cf. Balcerzak and Filipczak {\cite[Question 26]{BaFi11}}]
Let $\cI\in \{\cM,\cN\}$. Is it consistent that there exists a monotone
Borel hull on $\cI$ (with respect to $\cI$) but there is no such hull on
$\cS_\cI$ ? In particular, is it consistent that ${\rm add}(\cI)={\rm
  cof}(\cI)$ but there is no monotone Borel hull on $\cS_\cI$ ?
\end{problem}

It was noted in \cite[Proposition 7]{BaFi11} (see also Elekes and
M\'{a}th\'{e} \cite[Theorem 2.4]{ElMa09}) that ${\rm add}(\cI)={\rm
  cof}(\cI)$ implies that there exists a monotone Borel hull on $\cI$ (with
respect to $\cI$). It appears that was the only situation in which the
positive result of this kind was known. Using a finite support iteration of
ccc forcing notions we will show in this section that, consistently, we may
have ${\rm add}(\cI)<{\rm cof}(\cI)$ (for $\cI\in\{\cN,\cM\}$) and yet there
are monotone hulls for $\cI$.  

\begin{definition}
\label{mhg}
Let $\cI$ be an ideal of subsets of $\can$.
\begin{enumerate}
\item We say that a family $\cB\subseteq {\rm Borel}(\can)\cap \cI$ is {\em
    an mhg--base for $\cI$} if\footnote{``mhg'' stands for ``monotone hull
      generating''} 
\begin{enumerate}
\item[(a)] $\cB$ is a basis for $\cI$, i.e., $(\forall A\in\cI)(\exists B\in
  \cB)(A\subseteq B)$, and
\item[(b)] if $\langle B_i:i<\omega_1\rangle$ is a sequence of elements of
  $\cB$, then for some $i<j<\omega_1$ we have $B_i\subseteq B_j$.
\end{enumerate}
\item Let $\alpha^*,\beta^*$ be limit ordinals. {\em An
    $\alpha^*\times\beta^*$--base for $\cI$} is a sequence $\langle
  B_{\alpha,\beta}:\alpha<\alpha^*\ \&\ \beta<\beta^*\rangle$ of Borel sets 
    from $\cI$ such that it forms a basis for $\cI$ (i.e., (a) above holds)
    and
\begin{enumerate}
\item[(c)] for each $\alpha_0,\alpha_1<\alpha^*$, $\beta_0,\beta_1<\beta^*$
  we have
\[B_{\alpha_0,\beta_0}\subseteq B_{\alpha_1,\beta_1}\quad \Leftrightarrow\quad
\alpha_0\leq \alpha_1\ \&\ \beta_0\leq\beta_1.\]
\end{enumerate}
\end{enumerate}
\end{definition}

\begin{proposition}
\label{propA}
Assume that $\langle B_{\alpha,\beta}:\alpha<\alpha^*\ \&\
\beta<\beta^*\rangle$ is an $\alpha^*\times\beta^*$--base for $\cI$. Then:
\begin{enumerate}
\item[(i)] $B_{\alpha,\beta}\neq B_{\alpha',\beta'}$ whenever
  $(\alpha,\beta)\neq (\alpha',\beta')$, $\alpha,\alpha' <\alpha^*$,
  $\beta,\beta'<\beta^*$.
\item[(ii)] $\{B_{\alpha,\beta}:\alpha<\alpha^*\ \&\ \beta<\beta^*\}$ is an
  mhg--base for $\cI$.
\item[(iii)] ${\rm add}(\cI)=\min\{\cf(\alpha^*),\cf(\beta^*)\}$ and ${\rm
    cof}(\cI)=\max\{\cf(\alpha^*),\cf(\beta^*)\}$.
\end{enumerate}
\end{proposition}

\begin{proof}
Straightforward.
\end{proof}

\begin{proposition}
\label{gethull}
Suppose that an ideal $\cI$ on $\can$ has an mhg--base $\cB\subseteq {\rm
  Borel}(\can)\cap \cI$. Then there exists a monotone hull operation
$\psi:\cI\longrightarrow {\rm Borel}(\can)\cap \cI$ on $\cI$. If,
additionally, $\cB\subseteq \Pi^0_\xi$, $\xi<\omega_1$, then $\psi$ can be
taken to have values in $\Pi^0_\xi$. 
\end{proposition}

\begin{proof}
For a set $A\in\cI$ let $\cS_A$ be the family of all sequences
$\bar{B}=\langle B_i:i<\gamma\rangle\subseteq\cB$ satisfying
\begin{enumerate}
\item[$(*)_1$] $(\forall i<\gamma)(A\subseteq B_i)$ and 
\item[$(*)_2$] $(\forall i<j<\gamma)(B_i\nsubseteq B_j)$.
\end{enumerate}
Note that for each $\bar{B}\in\cS_A$ we have $\lh(\bar{B})<\omega_1$ (by
\ref{mhg}(1)(b) and $(*)_2$). Clearly, every $\trianglelefteq$--increasing
chain of elements of $\cS_A$ has a $\trianglelefteq$--upper bound in
$\cS_A$, so we may choose $\bar{B}_A=\langle B^A_i:i<\gamma_A\rangle\in
\cS_A$ which has no proper extension in $\cS_A$. Put
$\psi(A)=\bigcap\limits_{i<\gamma_A} B^A_i$. Plainly, $A\subseteq\psi(A)
\in\cI$ and $\psi(A)$ is a Borel set, and if $\cB\subseteq\Pi^0_\xi$ then
also $\psi(A)\in\Pi^0_\xi$.  

\begin{claim}
\label{clx}
$\psi(A)=\bigcap\{B\in\cB:A\subseteq B\}$  
\end{claim}

\begin{proof}[Proof of the Claim]
By $(*)_1$ we see that $\psi(A)\supseteq \bigcap\{B\in\cB:A\subseteq
B\}$. To show the converse inclusion suppose $B\in\cB$, $A\subseteq B$. By
the choice of $\bar{B}_A$ we know that $\bar{B}_A\conc \langle
B\rangle\notin\cS_A$ and hence $B^A_i\subseteq B$ for some
$i<\gamma_A$. Consequently $\psi(A)\subseteq B$. 
\end{proof}
It follows from the above claim that $A_1\subseteq A_2\in\cI$ implies
$\psi(A_1)\subseteq \psi(A_2)$. 
\end{proof}

Bartoszy\'nski and Kada \cite{BaKa05} showed that for any
$\sigma$--directed partial order $Q$ there is a ccc forcing notion
$\bbP$ such that
\[\forces_{\bbP}\mbox{`` $\cM$ has a basis order isomorphic
  to $Q$ with respect to set--inclusion ''.}\]
A parallel result for $\cN$ was given by Burke and Kada
\cite{BuKa04}. These theorems imply that for uncountable cardinals
$\kappa$ and $\lambda$ we may force that $\cM$ has a
$\kappa\times\lambda$--basis, and we may also force that 
$\cN$ has a $\kappa\times\lambda$--basis. The corresponding forcing
notions (for both cases) were essentially versions of  ``FS iterations
with partial memories'' used in Shelah \cite{Sh:592, Sh:546, Sh:619},
Mildenberger and Shelah \cite{MdSh:684} and Shelah and Thomas
\cite{ShTh:524}.  We will use explicitly the method of ``FS iterations
with partial memories'' to construct a model in which {\em both\/}
ideals have $\kappa\times\lambda$--bases.  

\begin{theorem}
\label{modelwithmhg}
Let $\kappa,\lambda$ be cardinals of uncountable cofinality, $\kappa\leq
\lambda$. There is a ccc forcing notion $\bQ^{\kappa,\lambda}$ of size
$\lambda^{\aleph_0}$ such that
\[\begin{array}{ll}
\forces_{\bQ^{\kappa,\lambda}}&\mbox{`` the meager ideal $\cM$ has a
  $\kappa\times\lambda$--basis consisting of $\Sigma^0_2$ sets, and}\\
&\mbox{\ \ the null ideal $\cN$ has a $\kappa\times\lambda$--basis
  consisting of $\Pi^0_2$ sets ''.}
\end{array}\] 
\end{theorem}

\begin{proof}
The forcing notion $\bQ^{\kappa,\lambda}$ will be obtained by means of
finite support iteration of ccc forcing notions. The iterands will be
products of the Amoeba for Category $\bbB$ and Amoeba for Measure $\bbA$ but
{\em considered over partial sub-universes only}. 

We will use the notation and some basic facts stated in the third
section of \cite{ShTh:524}.  

Let us recall the forcings $\bbA$ and $\bbB$ used as iterands. 
\begin{itemize}
\item A condition in $\bbA$ is a tree $T\subseteq\fs$ such that
$\mu([T])>\frac{1}{2}$ and $\mu([t]\cap [T])>0$ for all $t\in T$. The order
$\leq_{\bbA}$ of $\bbA$ is the reverse inclusion.  
\item A condition in $\bbB$ is a pair $(n,T)$ such that $n\in\omega$,
  $T\subseteq\fs$ is a tree with no maximal nodes and $[T]$ is a nowhere
  dense subset of $\can$. The order $\leq_{\bbB}$ of $\bbB$ is given by:\\ 
$(n,T)\leq_{\bbB} (n',T')$ if and only if $n\leq n'$, $T\subseteq T'$ and
$T\cap {}^{\textstyle n} 2= T'\cap {}^{\textstyle n} 2$.
\end{itemize}
Both $\bbA$ and $\bbB$ are (nice definitions of) ccc forcing notions, $\bbB$
is $\sigma$--centered and if $\bV'\subseteq\bV''$ are universes of set
theory then $\bbA^{\bV'}$ is still ccc in $\bV''$. We will use the following
immediate properties of these forcing notions. 
\begin{enumerate}
\item[$(\circledast)_1$] If $G\subseteq\bbA$ is generic over $\bV$,
  $F=\bigcap\{[T]: T\in G\}$, then $F$ is a closed subset of $\can$,
  $\mu(F)=\frac{1}{2}$ and $F$ is disjoint from every Borel null set coded in
  $\bV$. Hence the set $F^*=\{x\in\can: (\forall y\in F)(\exists^\infty
  n)(x(n)\neq y(n))\}$ is a null $\Pi^0_2$ set and it includes all Borel
  null sets coded in $\bV$.\\ 
Let $\name{F}_\bbA,\name{F}_\bbA^*$ be $\bbA$--names for the sets $F,F^*$, 
respectively. 
\item[$(\circledast)_2$] If $G\subseteq \bbB$ is generic over $\bV$, $F=
  \bigcup\{[T]:(\exists n)((n,T)\in G)\}$, then $F$ is a closed nowhere
  dense subset of $\can$. Letting $F^*=\{x\in\can: (\exists y\in F)(
  \forall^\infty n)(x(n)=y(n))\}$ we get a meager $\Sigma^0_2$ set including
  all Borel meager sets coded in $\bV$.\\
Let $\name{F}_\bbB,\name{F}_\bbB^*$ be $\bbB$--names for the sets $F,F^*$, 
respectively. 
\item[$(\circledast)_3^a$] If $T\in\bbA$, $t\in T$, then there is
  $T'\geq_\bbA T$ such that $T'\forces_\bbA
  [t]\cap\name{F}_\bbA\neq\emptyset$.  
\item[$(\circledast)_3^b$] If $T\in\bbA$, $n\in\omega$, then there is $N>n$
  such that for each $\nu\in {}^{[n,N)}2$ there is $T'\geq_\bbA T$ with
$T'\forces_\bbA (\forall y\in\name{F}_\bbA)(y\rest [n,N)\neq \nu)$.
\item[$(\circledast)_4^a$] If $(n,T)\in\bbB$, $t\in T$, $\lh(t)\leq n$,
  $m_1>m_0\geq n$ and $\nu\in {}^{[m_0,m_1)}2$, then there are
  $(n',T')\geq_\bbB (n,T)$ and $s\in T'$ such that $t\vartriangleleft s$ and
  $s\rest [m_0,m_1)=\nu$ (and $(n',T')\forces_\bbB [s]\cap\name{F}_\bbB\neq
  \emptyset$). 
\item[$(\circledast)_4^b$] If $(n,T)\in \bbB$, $m_0<\omega$, then there are
  $m_1>m_0$ and $\nu\in {}^{[m_0,m_1)} 2$ and $(n',T')\geq_\bbB (n,T)$
  such that $(n',T')\forces_\bbB (\forall y\in\name{F}_\bbB)(y\rest
  [m_0,m_1)\neq \nu)$.  
\end{enumerate}
Fix an ordinal $\gamma$ and a bijection $\pi:\kappa\times\lambda
\stackrel{\rm onto}{\longrightarrow}\gamma$ such that 
\[\alpha_0\leq \alpha_1<\kappa\ \&\ \beta_0\leq \beta_1<\lambda\quad
\Rightarrow\quad \pi(\alpha_0,\beta_0)\leq \pi(\alpha_1,\beta_1).\]
For $i=\pi(\alpha_1,\beta_1)$ let $a_i=\{\pi(\alpha_0,\beta_0):\alpha_0 \leq
\alpha_1\ \&\ \beta_0\leq \beta_1\}\setminus \{i\}$. We say that a set
$b\subseteq \gamma$ is {\em closed\/} if $a_i\subseteq b$ for all $i\in
b$. It follows from our choice of $\pi$ that for each $i<\gamma$ we have 
\begin{enumerate}
\item[$(\circledast)_5$] $a_i\subseteq i$ and the sets $a_i,i,a_i\cup\{i\}$
  are closed.
\end{enumerate}
Now, by induction we define $\langle\bbP_i,\name{\bQ}_i, \name{F}^0_i,
\name{F}^1_i, \name{F}^\bbA_i,\name{F}^\bbB_i:i<\gamma\rangle$ and
$\bbP^*_b$ for closed $b\subseteq\gamma$ simultaneously proving the
correctness of the definition and the desired properties listed
below.\footnote{See \cite[3.1--3.7]{ShTh:524} for the order in which these
  should be shown.}  
\begin{enumerate}
\item[$(\circledast)_6$] $\langle\bbP_j,\name{\bQ}_i:j\leq\gamma,
  i<\gamma\rangle$ is a finite support iteration of ccc forcing notions. 
\item[$(\circledast)_7$] $\bbP^*_b=\big\{p\in\bbP_\gamma:\supp(p)\subseteq
  b\ \&\ p(i)$ is a $\bbP^*_{a_i}$--name (for a member of $\name{\bQ}_i$) for
  each $i\in\supp(p)\big\}$. 
\item[$(\circledast)_8$] $\bbP^*_b$ is a complete suborder of $\bbP_\gamma$,
  $\bbP^*_{a_i\cup\{i\}}$ is isomorphic with the composition $\bbP_{a_i}^*
  *\name{\bQ}_i$. 
\item[$(\circledast)_9$] $\name{\bQ}_i$ is a $\bbP^*_{a_i}$--name for the
  product\footnote{Since $\bbB^{\bV^{\bbP^*_{a_i}}}$ is $\sigma$--centered
    we know that the product is ccc.} $\bbA\times\bbB$.
\item[$(\circledast)_{10}$] $\nF^0_i,\nF^1_i,\nF^\bbA_i,\nF^\bbB_i$ are
  $\bbP^*_{a_i\cup\{i\}}$--names for the sets $\nF_\bbA,\nF_\bbB,\nF^*_\bbA,
  \nF^*_\bbB$ added by the forcings at the last coordinate of
  $\bbP^*_{a_i\cup\{i\}}\simeq \bbP^*_{a_i}*(\bbA\times\bbB)$. 
\item[$(\circledast)_{11}$] 
\begin{enumerate}
\item[(a)] $\bbP^*_i$ is a dense subset of $\bbP_i$ (for $i\leq\gamma$).
\item[(b)] If $q\in\bbP^*_\gamma$, then $q\rest b\in \bbP^*_b$.
\item[(c)] If $p,q\in\bbP^*_\gamma$, $p\leq q$ and $i\in\supp(q)$ then
  $p\rest a_i\leq_{\bbP^*_{a_i}} q\rest a_i$ and $q\rest
  a_i\forces_{\bbP^*_{a_i}} p(i)\leq q(i)$. 
\item[(d)] If $q\in\bbP^*_\gamma$, $p\in\bbP^*_b$ and $p\leq q$, then
  $p\leq_{\bbP^*_b} q\rest b$.
\item[(e)] If $q\in\bbP^*_b$, $p\in\bbP_\gamma^*$, $p\rest b\leq_{\bbP^*_b}
  q$ and $r$ is defined by 
\[r(\xi)=\left\{\begin{array}{ll}
q(\xi)&\mbox{ if }\xi\in b,\\
p(\xi)&\mbox{ otherwise}
\end{array}\right.\qquad\mbox{ for }\xi<\gamma\]
then $r\in\bbP^*_\gamma$ and $r\geq q$, $r\geq p$. 
\end{enumerate}
\end{enumerate}

Also,
\begin{enumerate}
\item[$(\circledast)_{12}$] if $\name{\tau}$ is a canonical\footnote{i.e.,
    determined in a standard way by a sequence of maximal antichains}
  $\bbP^*_\gamma$--name for a member of $\can$, then $\name{\tau}$ is a
  $\bbP^*_{a_i}$--name for some $i<\gamma$. 
\end{enumerate}
[Why? Note that if $(\alpha_n,\beta_n)\in\kappa\times\lambda$, $n<\omega$,
then there is $(\alpha^*,\beta^*)\in\kappa\times\lambda$ such that
$\alpha_n\leq\alpha^*$, $\beta_n\leq\beta^*$ for all $n<\omega$.]

The main technical point of our argument is given in the following
observation. 

\begin{enumerate}
\item[$(\circledast)_{13}$] Suppose $i,j<\gamma$, $i\notin a_j$, $j\notin
  a_i$, $i\neq j$, $\ell\in \{0,1\}$. Assume that $p\in\bbP^*_\gamma$,
  $\eta\in {}^{\textstyle n} 2$, $n<\omega$ and $p\forces_{\bbP^*_\gamma}
  [\eta]\cap\nF^\ell_i\neq \emptyset$. Then there are $\nu\in {}^{[n,N)}2$,
  $n<N<\omega$ and $q\geq_{\bbP^*_\gamma} p$ such that
\[q\forces_{\bbP^*_\gamma}\mbox{`` }[\eta\conc\nu]\cap\nF^\ell_i\neq
\emptyset \mbox{ and }\big(\forall y\in\nF^\ell_j\big)\big(y\rest
[n,N)\neq\nu\big)\mbox{ ''.}\]      
\end{enumerate}
[Why? Let us provide detailed arguments for $\ell=0$. By $(\circledast)_3^b+
(\circledast)_9+(\circledast)_{11}$ we may find $N>n$ and a condition
$p_0'\in\bbP^*_{a_j}$ such that $p_0'\geq p\rest a_j$ and 
\[\begin{array}{ll}
p_0'\forces_{\bbP^*_{a_j}}&\mbox{`` for each $\nu\in {}^{[n,N)}2$ there is
  $p_j\geq_{\name{\bQ}_j} p(j)$ such that}\\
&\quad p_j\forces_{\name{\bQ}_j}(\forall y\in \nF_\bbA)(y\rest [n,N)\neq \nu)
\mbox{ ''.}
\end{array}\] 
Let $p_0\in\bbP^*_\gamma$ be such that $p_0(\xi)=p_0'(\xi)$ for $\xi\in
a_j$ and $p_0(\xi)=p(\xi)$ otherwise (see $(\circledast)_{11}(e)$; so $p_0$
is a common extension of $p_0'$ and $p$). Note that $p_0(j)=p(j)$. Use
$(\circledast)^a_3$ to choose $\nu\in {}^{[n,N)}2$ and a condition
$p_1'\in\bbP^*_{a_i\cup\{i\}}$ such that $p_1'\geq p_0\rest (a_i\cup\{i\})$
and $p_1'\forces_{\bbP^*_{a_i\cup\{i\}}} [\eta\conc \nu]\cap\nF^0_i\neq
\emptyset$. Let $p_1\in\bbP^*_\gamma$ be such that $p_1(\xi)=p_1'(\xi)$ if
$\xi\in a_i\cup\{i\}$ and $p_1(\xi)=p_0(\xi)$ otherwise. Then $p_1$ is
stronger than both $p_1'$ and $p_0$, and $p_1(j)=p_0(j)=p(j)$. Hence
\[p_1\rest a_j\forces_{\bbP^*_{a_j}}\mbox{`` there is }p_j \geq_{\name{\bQ}_j}
p_1(j) \mbox{ such that }p_j\forces_{\name{\bQ}_j} (\forall y\in
\nF_\bbA)(y\rest [n,N)\neq \nu) \mbox{ ''.} \]
Let $q(j)$ be a $\bbP^*_{a_j}$--name for a $p_j$ as above and let
$q(\xi)=p_1(\xi)$ for $\xi\neq j$. Clearly $q\in\bbP^*_\gamma$ and 
$q\rest (a_j\cup\{j\})\forces_{\bbP^*_{a_j\cup\{j\}}}(\forall y\in
\nF^0_j)(y\rest [n,N)\neq \nu)$, and (as $q\rest (a_i\cup\{i\})=p_1\rest
(a_i\cup\{i\})= p_1'$)  $q\rest (a_i\cup\{i\})\forces_{\bbP^*_{a_i\cup
    \{i\}}} [\eta\conc\nu]\cap\nF^0_i\neq\emptyset$. Using $(\circledast)_8+
(\circledast)_{10}+(\circledast)_{11}$ we get that the condition $q$ is as
required. If $\ell=1$ then the arguments are similar, but instead of 
$(\circledast)^a_3,(\circledast)^b_3$ we use
$(\circledast)^a_4,(\circledast)^b_4$.]
\medskip

For $\alpha<\kappa$, $\beta<\lambda$ let $\name{B}^\bbA_{\alpha,\beta}
=\nF^\bbA_{\pi(\alpha,\beta)}$ and $\name{B}^\bbB_{\alpha,\beta}
=\nF^\bbB_{\pi(\alpha,\beta)}$. Immediately from $(\circledast)_{12}+
(\circledast)_1 + (\circledast)_2$ we conclude that 
\begin{enumerate}
\item[$(\circledast)_{14}$] $\forces_{\bbP^*_\gamma}\mbox{`` }
  \{\name{B}_{\alpha,\beta}^\bbA:\alpha<\kappa\ \&\ \beta<\lambda\}$ is a
  basis for $\cN$ and

\qquad $\{\name{B}_{\alpha,\beta}^\bbB:\alpha<\kappa\ \&\ \beta<\lambda\}$
is a basis for $\cM$ ''
\end{enumerate}
and
\begin{enumerate}
\item[$(\circledast)_{15}$] if $\alpha_0\leq\alpha_1<\kappa$,
  $\beta_0\leq\beta_1<\lambda$, $(\alpha_0,\beta_0)\neq (\alpha_1,\beta_1)$,
  then 
\[\forces_{\bbP^*_\gamma}\mbox{`` }
\name{B}_{\alpha_0,\beta_0}^\bbA\subsetneq \name{B}_{\alpha_1,\beta_1}^\bbA
\ \&\ \name{B}_{\alpha_0,\beta_0}^\bbB\subsetneq
\name{B}_{\alpha_1,\beta_1}^\bbB\mbox{ ''.}\]
\end{enumerate}
Also 
\begin{enumerate}
\item[$(\circledast)_{16}$] if $\alpha_0,\alpha_1<\kappa$,
  $\beta_0,\beta_1<\lambda$ and $\neg\big(\alpha_0\leq \alpha_1\ \&\ \beta_0
  \leq\beta_1\big)$ then 
\[\forces_{\bbP^*_\gamma}\mbox{`` }
\name{B}_{\alpha_0,\beta_0}^\bbA\nsubseteq \name{B}_{\alpha_1,\beta_1}^\bbA
\ \&\ \name{B}_{\alpha_0,\beta_0}^\bbB\nsubseteq
\name{B}_{\alpha_1,\beta_1}^\bbB\mbox{ ''.}\]
\end{enumerate}
[Why? If $\alpha_1\leq\alpha_0$ and $\beta_1\leq \beta_0$, then
$(\circledast)_{15}$ applies, so we may assume additionally 
$\neg\big(\alpha_1\leq\alpha_0\ \&\ \beta_1\leq \beta_0\big)$. Then our
assumptions on $\alpha_0,\alpha_1,\beta_0,\beta_1$ mean that, letting
$j=\pi(\alpha_0,\beta_0)$ and $i=\pi(\alpha_1,\beta_1)$, we have 
$i\notin a_j$, $j\notin a_i$, $i\neq j$. So using $(\circledast)_{13}$ for 
$\ell=0$ we easily build a $\bbP^*_\gamma$--name $\name{\eta}$ for a member
of $\can$ such that  
\[\forces_{\bbP^*_\gamma}\mbox{`` }\name{\eta}\in [\nF^0_i]\subseteq
\can\setminus \nF^\bbA_i=\can\setminus\name{B}^\bbA_{\alpha_1,\beta_1}\ \&\ 
\name{\eta}\in \nF^\bbA_j=\name{B}^\bbA_{\alpha_0,\beta_0}\mbox{ ''.}\] 
Similarly, using $(\circledast)_{13}$ for $\ell=1$ and interchanging the
role of $i$ and $j$ we may construct a $\bbP^*_\gamma$--name $\name{\eta}'$
such that $\forces_{\bbP^*_\gamma}\mbox{`` }\name{\eta}'\in
\name{B}^\bbB_{\alpha_0,\beta_0}\setminus\name{B}^\bbB_{\alpha_1,
  \beta_1}\mbox{ ''}$. ]
\medskip

Finally we note that $\bbP^*_\gamma$ has a dense subset of size
$\lambda^{\aleph_0}$, so we may choose it as our desired forcing
$\bQ^{\kappa,\lambda}$. 
\end{proof}

\begin{remark}
In a manner similar to our proof of $(\circledast)_{13}$ above one may argue
for the following stronger property. 
\begin{enumerate}
\item[$(\circledast)^2_{13}$]  Suppose $i,j<\gamma$, $i\notin a_j$, $j\notin
  a_i$, $i\neq j$, $\ell\in \{0,1\}$. Assume that $p\in\bbP^*_\gamma$,
  $\eta\in {}^{\textstyle n} 2$, $n<\omega$ and $p\forces_{\bbP^*_\gamma}
  [\eta]\cap\nF^\ell_i\neq \emptyset$. Then there are $\nu_0,\nu_1
\in {}^{[n,N)}2$, $n<N<\omega$ and $q\geq_{\bbP^*_\gamma} p$ such that
$\nu_0\neq \nu_1$ and 
\[q\forces_{\bbP^*_\gamma}\mbox{`` }[\eta\conc\nu_0]\cap\nF^\ell_i
\neq\emptyset\neq [\eta\conc\nu_1]\cap\nF^\ell_i\mbox{ and }
\big(\forall y\in\nF^\ell_j\big)\big(y\rest
[n,N)\notin\{\nu_0,\nu_1\}\big)\mbox{ ''.}\] 
\end{enumerate}
Then, if $i=\pi(\alpha_0,\beta_0)$, $j=\pi(\alpha_1,\beta_1)$, $i\notin
a_j$, $j\notin a_i$ and $i\neq j$, we may use this property to construct
$\bbP^*_\gamma$--names $\name{T}^\bbA$ and $\name{T}^\bbB$ for perfect
subtrees of $\fs$ such that 
\[\forces_{\bbP^*_\gamma}\mbox{`` }
[\name{T}^\bbA]\subseteq \name{B}_{\alpha_0,\beta_0}^\bbA\setminus
\name{B}_{\alpha_1,\beta_1}^\bbA
\ \mbox{ and }\ [\name{T}^\bbB]\subseteq \name{B}_{\alpha_0,\beta_0}^\bbB
\setminus\name{B}_{\alpha_1,\beta_1}^\bbB\mbox{ ''.}\]
Also $(\circledast)_{15}$ can easily strengthen to 
\begin{enumerate}
\item[$(\circledast)^+_{15}$]  if $\alpha_0\leq\alpha_1<\kappa$,
  $\beta_0\leq\beta_1<\lambda$, $(\alpha_0,\beta_0)\neq (\alpha_1,\beta_1)$,
  then 
\[\forces_{\bbP^*_\gamma}\mbox{`` both }
\name{B}_{\alpha_1,\beta_1}^\bbA\setminus \name{B}_{\alpha_0,\beta_0}^\bbA
\ \mbox{ and }\ \name{B}_{\alpha_1,\beta_1}^\bbB\setminus
\name{B}_{\alpha_0,\beta_0}^\bbB\mbox{ are uncountable ''.}\]
\end{enumerate}
Consequently, in $\bV^{\bbP^*_\gamma}$, the $\kappa\times\lambda$--bases
$\{\name{B}^\bbA_{\alpha,\beta}: \alpha<\kappa,\ \beta<\lambda\}$ and 
$\{\name{B}^\bbB_{\alpha,\beta}: \alpha<\kappa,\ \beta<\lambda\}$ have the
additional property that 
\[\forces_{\bbP^*_\gamma}\mbox{`` }\alpha_0>\alpha_1\ \vee \ \beta_0>
\beta_1\quad \Rightarrow\quad \big|\name{B}_{\alpha_0,\beta_0}^\bbA
\setminus\name{B}_{\alpha_1,\beta_1}^\bbA\big|=\big|
\name{B}_{\alpha_0,\beta_0}^\bbB\setminus
\name{B}_{\alpha_1,\beta_1}^\bbB\big|=2^{\aleph_0}\mbox{ ''.}\]
This is used is Ros{\l}anowski and Shelah \cite{RoSh:1022}.
\end{remark}

\begin{corollary}
\label{cor4}
It is consistent that
\begin{itemize}
\item ${\rm add}(\cN)={\rm add}(\cM)<{\rm cof}(\cN)={\rm
    cof}(\cM)=2^\omega$ (and hence the ideals $\cM,\cN$ do not poses tower
  bases) , and
\item there is a monotone $\Pi^0_3$ hull operation on $\cM$ with respect
  to $\cM$, and  
\item there is a monotone $\Pi^0_2$ hull operation on $\cN$ with respect to
  $\cN$, and  
\item there is a monotone $\Pi^0_3$ hull operation on $\cM\cap \cN$ with
  respect to $\cM\cap\cN$. 
\end{itemize}
\end{corollary}

\begin{proof}
Start with a universe satisfying CH and use the forcing given by
Theorem \ref{modelwithmhg} for $\kappa=\aleph_1$ and
$\lambda=\aleph_2$. Propositions \ref{gethull} and \ref{propA} imply
that the resulting model is as required. 
\end{proof}

\begin{remark}
In Theorem \ref{modelwithmhg} we obtained a universe of set theory in which
both $\cN$ and $\cM$ have bases that are (with respect to the inclusion)
order isomorphic to $\kappa\times\lambda$. We may consider any partial order
$(S,\sqsubseteq)$ such that 
\begin{enumerate}
\item[(a)] $|S|=\lambda$ and $(S,\sqsubseteq)$ is well founded, and
\item[(b)] every countable subset of $S$ has a common $\sqsubseteq$--upper
  bound. 
\end{enumerate}
Then by a very similar construction we get a forcing extension in which both
$\cN$ and $\cM$ have bases order isomorphic to $(S,\sqsubseteq)$. If
additionally 
\begin{enumerate}
\item[(c)] for every sequence $\langle s_i:i<\omega_1\rangle\subseteq S$
  there are $i<j<\omega_1$ such that $s_i\sqsubseteq s_j$,
\end{enumerate}
then those bases will be mhg. (Note that forcings with the Knaster property
preserve the demand described in (c).) 
\end{remark}


\end{document}